\documentclass[a4paper,10pt]{article}
\usepackage[utf8x]{inputenc}
\usepackage{tracefnt,amsmath,tabu,array}
\usepackage{amssymb,graphicx,setspace,amsfonts,amsbsy}
\usepackage{pifont,latexsym,ifthen,amsthm,rotating,calc,textcase,booktabs}

\newtheorem{theorem}{Theorem}[section]
\newtheorem{lemma}[theorem]{Lemma}
\newtheorem{corollary}[theorem]{Corollary}
\newtheorem{proposition}[theorem]{Proposition}
\newtheorem{definition}[theorem]{Definition}

\newtheorem{remark}[theorem]{Remark}
\newcommand{\filledbox}{\leavevmode
  \hbox to.77778em{%
  \hfil\vbox to.675em{\hrule width.6em height.6em}\hfil}}

\newcommand{\Rm}{{\mathbb R}}

\newcommand{\eps}{\varepsilon}

\begin{document}
\tabulinesep=1.0mm
\title{Scattering of solutions to the defocusing energy sub-critical semi-linear wave equation in 3D\footnote{MSC classes: 35L71, 35L05}}


\author{Ruipeng Shen\\
Centre for Applied Mathematics\\
Tianjin University\\
Tianjin, China}

\maketitle

\begin{abstract}
  In this paper we consider a semi-linear, energy sub-critical, defocusing wave equation $\partial_t^2 u - \Delta u = - |u|^{p -1} u$ in the 3-dimensional space 
 with $p \in [3,5)$. We prove that if initial data $(u_0, u_1)$ are radial so that $\|\nabla u_0\|_{L^2 (\Rm^3; d\mu)}, \|u_1\|_{L^2 (\Rm^3; d\mu)} \leq \infty$, where $d \mu = (|x|+1)^{1+2\eps}$ with $\eps > 0$, then the corresponding solution $u$ must exist for all time $t \in {\mathbb R}$ and scatter. The key ingredients of the proof include a transformation $\mathbf{T}$ so that $v = \mathbf{T} u$ solves the equation $v_{\tau \tau} - \Delta_y v = - \left(\frac{|y|}{\sinh |y|}\right)^{p-1} e^{-(p-3)\tau} |v|^{p-1}v$ with a finite energy, and a couple of global space-time integral estimates regarding a solution $v$ as above. 
\end{abstract}

\section{Introduction}

The defocusing semi-linear wave equation 
\[
 \left\{\begin{array}{ll} \partial_t^2 u - \Delta u = - |u|^{p-1}u, & (x,t) \in \Rm^3 \times \Rm; \\
 u(\cdot, 0) = u_0; & \\
 u_t (\cdot,0) = u_1 & \end{array}\right.\quad (CP1)
\]
has been extensively studied in the past few decades. This problem is locally well-posed if initial data $(u_0,u_1)$ are contained in the critical Sobolev space $\dot{H}^{s_p} \times \dot{H}^{s_p-1}(\Rm^3)$ with $s_p \doteq 3/2 - 2/(p-1)$. Please see \cite{ls} for more details on the local theory. Suitable solutions also satisfy an energy conservation law:
\[
 E(u, u_t) = \int_{\Rm^3} \left(\frac{1}{2}|\nabla u(\cdot, t)|^2 +\frac{1}{2}|u_t(\cdot, t)|^2 + \frac{1}{p+1}|u(\cdot,t)|^{p+1}\right)\,dx = \hbox{Const}.
\]
The problem of global existence and scattering is much more difficult. In the energy critical case $p=5$, M. Grillakis \cite{mg1} proved that any solution with initial data in the space $\dot{H}^1 \times L^2(\Rm^3)$ must scatter in both two time directions. In other words, the asymptotic behaviour of any solution mentioned above resembles that of a free wave. It is conjectured that a similar result holds for other exponents $p$ as well: Any solution to (CP1) with initial data $(u_0,u_1) \in \dot{H}^{s_p} \times \dot{H}^{s_p-1}$ must exist for all time $t \in \Rm$ and scatter in both two time directions. This conjecture has not been proved yet, as far as the author knows, in spite of some progress:
\begin{itemize}
 \item It has been proved that if a radial solution $u$ with a maximal lifespan $I$ satisfies an a priori estimate
 \begin{equation}
  \sup_{t \in I} \left\|(u(\cdot,t), u_t(\cdot, t))\right\|_{\dot{H}^{s_p} \times \dot{H}^{s_p-1} (\Rm^3)} < +\infty, \label{uniform UB}
 \end{equation}
 then $u$ is a global solution in time and scatters.  The proof uses the standard compactness-rigidity argument, where the radial assumption plays a crucial role in the rigidity part. The details can be found in \cite{km} for $p>5$, \cite{shen2} for $3<p<5$ and \cite{cubic3dwave} for $1+\sqrt{2}<p\leq 3$.  The author would also like to mention that the same result still holds in the non-radial case if $p>5$, see \cite{kv2}. Please note that our assumption \eqref{uniform UB} is automatically true in the energy critical case $p=5$, thanks to the conservation law of energy. When $p$ is other than $5$, however, nobody has ever found a way to actually prove this a priori estimate without additional assumptions on initial data.
 \item In the energy sub-critical case, the scattering result can be proved via conformal conservation laws if initial data satisfy an additional regularity-decay condition
  \begin{equation} \label{condition1}
    \int_{\Rm^3} \left[(|x|^2+1) (|\nabla u_0 (x)|^2 + |u_1(x)|^2) + |u_0(x)|^2 \right] dx < \infty.
  \end{equation}
See \cite{conformal2, conformal} for more details. Please pay attention that the radial assumption is not necessary in this argument.  
\end{itemize}

\paragraph{Main Result} In this work we assume that initial data are still radial but satisfy a weaker decay condition than \eqref{condition1} and prove that the corresponding solution to (CP1) scatters. Let us first introduce our main theorem

\begin{theorem} \label{main1}
Assume that $A, \eps$ are positive constants and $3 \leq p<5$. Let $(u_0, u_1) \in \dot{H}^1 \times L^2$ be radial initial data so that 
\begin{align*}
  &\|\nabla u_0\|_{L^2 (\Rm^3; d\mu)}, \|u_1\|_{L^2 (\Rm^3; d\mu)} \leq A,& &d \mu = (|x|+1)^{1+2\eps} dx.&
\end{align*}
Then the corresponding solution $u$ to (CP1) scatters in both two time directions with 
\[ 
 \|u\|_{L^{2(p-1)} L^{2(p-1)} (\Rm \times \Rm^3)} \leq C(A, \eps, p) < \infty. 
\]
Here the upper bound $C(A,\eps,p)$ are solely determined by the values of $A$, $\eps$ and $p$.
\end{theorem}
\noindent Here are some remarks regarding the initial data in the main theorem. 
\begin{remark}
The initial data $(u_0, u_1)$ satisfy the inequality 
\begin{align*}
 \int_{\Rm^3} \!\left(|\nabla u_0|^\frac{3}{2} + |u_1|^\frac{3}{2} \right) dx \leq &2\left[\int_{\Rm^3} \!\left(|\nabla u_0|^2 + |u_1|^2 \right)(1+|x|)^{1+2\eps}\, dx\right]^{3/4}  \left[\int_{\Rm^3} \! (1+|x|)^{-3-6\eps}\, dx\right]^{1/4}\\
  \leq & C(A, \eps) < \infty. 
\end{align*}
In other words we have $(u_0,u_1) \in \dot{W}^{1, 3/2} \times L^{3/2}$. It immediately follows that $(u_0,u_1) \in \dot{H}^{s_p} \times \dot{H}^{s_p-1}$ by the Sobolev embedding $\dot{W}^{1, 3/2} \times L^{3/2} \hookrightarrow \dot{H}^{1/2} \times \dot{H}^{-1/2}$ and an interpolation. 
\end{remark}

\begin{remark} \label{integral in r}
The radial assumption implies that the initial data $(u_0, u_1)$ satisfy
 \[
 \int_0^\infty \left(|\partial_r u_0(r)|^2 + |u_1(r)|^2\right) r^{3+2\eps} dr \leq (1/4\pi) A^2.
\]
\end{remark}

\begin{remark} \label{point-wise es and energy}
 Any pair $(u_0, u_1)$ as in Theorem \ref{main1} comes with a finite energy
 \begin{align*}
  E(u_0, u_1) = \int_{\Rm^3} \left[\frac{1}{2}|\nabla u_0(x)|^2 + \frac{1}{2} |u_1(x)|^2 + \frac{1}{p+1}|u_0(x)|^{p+1}\right] dx \leq C(A) < \infty.
 \end{align*}
 In addition, $u_0$ satisfies a point-wise estimate $ |u_0 (x)| \leq A |x|^{-1-\eps}$.
\end{remark}

\begin{proof}
By Remark \ref{integral in r} we have ($0<r_1<r_2<\infty$)
\begin{align}
 \left|u_0(r_1) - u_0 (r_2)\right| & \leq \int_{r_1}^{r_2} |\partial_r u_0(r)| dr \leq \left(\int_{r_1}^{r_2} |\partial_r u_0(r)|^2 r^{3+2\eps} dr\right)^{1/2} 
 \left(\int_{r_1}^{r_2} r^{-3-2\eps} dr\right)^{1/2}\nonumber \\
 &\leq A r_1^{-1-\eps}. \label{diff_limit}
\end{align}
Next we recall the point-wise estimate for radial $\dot{H}^1$ functions $|u_0(x)| \leq C \|u_0\|_{\dot{H}^1}|x|^{-1/2}$ as given in Lemma 3.2 of \cite{km},
make $r_2 \rightarrow \infty$ in the inequality \eqref{diff_limit} above and obtain a point-wise estimate $ |u_0 (x)| \leq A |x|^{-1-\eps}$. Furthermore, we can combine this point-wise estimate with the Sobolev embedding $\dot{H}^1 (\Rm^3) \hookrightarrow L^6 (\Rm^3)$ to conclude $\|u_0\|_{L^{p+1}(\Rm^3)} \leq C(A)$. This immediately gives a finite upper bound on the energy. 
\end{proof}

\paragraph{The idea} In order to prove the main theorem, we need to show the following step by step. 
\begin{itemize}
 \item The solution $u$ is defined for all time $t \in \Rm$.
 \item The function $v = \mathbf{T} u$ defined by ($t_0$ is a time to be determined)
 \[
   v(y, \tau)  = \frac{\sinh |y|}{|y|} e^\tau u \left( e^\tau \frac{\sinh |y|}{|y|}\cdot y, t_0 + e^\tau \cosh |y|\right), \quad (y,\tau) \in \Rm^3 \times \Rm
 \]
 solves the following non-linear wave equation with a finite energy
\[
 v_{\tau \tau} - \Delta_y v = - \left(\frac{|y|}{\sinh |y|}\right)^{p-1} e^{-(p-3)\tau} |v|^{p-1}v. \quad (\hbox{CP2})
\]
\item The solution $v$ satisfies a few space-time integral estimates. 
\item We rewrite the information about $v$ obtained in the previous step in term of $u$ and finally conclude $\|u\|_{L^{2(p-1)} L^{2(p-1)}(\Rm \times \Rm^3)} < \infty$. This is equivalent to the scattering of $u$, as shown in Subsection \ref{sec: equivalent condition of scattering}.
\end{itemize}

\noindent The transformation from $u$ to $v$ above is one of the key ingredients of our proof. Its validity can be verified by a basic calculation, as given in Section \ref{sec: transformation}. The author would also like to mention that the transformation can be constructed via two different routes:

\paragraph{Route 1} We can write $\mathbf{T} = \mathbf{T}_2 \circ \mathbf{T}_1$. Here $\mathbf{T}_1$ is a transformation from the set of functions defined on the forward light cone $\{(x,t): t-t_0>|x|\}$ to the set of functions defined on ${\mathbb H}^3 \times \Rm$, whose formula has been given by D. Tataru in the work \cite{tataru}:
\[
 (\mathbf{T}_1 u)(s, \Theta, \tau) = e^\tau u (e^\tau \sinh s \cdot \Theta, t_0 + e^\tau \cosh s).
\] 
Here $(s, \Theta) \in [0,\infty) \times {\mathbb S}^2$ are polar coordinates on the hyperbolic space ${\mathbb H}^3$. One can demonstrate the importance of this transformation by the fact 
\[ 
  (\partial_\tau^2 - \Delta_{{\mathbb H}^3} - 1) \circ \mathbf{T}_1 = e^{2\tau} \mathbf{T}_1 \circ (\partial_t^2 - \Delta).
\]
As a result, if $u$ is a solution to (CP1), then the function $v_1 = \mathbf{T}_1 u$ solves the non-linear shifted wave equation on ${\mathbb H}^3$ (See \cite{wavehyper, subhyper, energyhyper} for Strichartz estimates and local theory on this type of equations)
\begin{equation}
 \partial_\tau^2 v_1 - (\Delta_{{\mathbb H}^3} +1) v_1 = - e^{-(p-3)\tau} |v_1|^{p-1} v_1. \label{equation 102}
\end{equation}
Next we introduce the second transformation\footnote{we need to use the radial assumption on $v_1$ in the definition.} $(\mathbf{T}_2 v_1) (y,\tau) = \frac{\sinh |y|}{|y|} v_1(|y|, \tau)$, whose domain is the set of radial functions on ${\mathbb H}^3 \times \Rm$ and whose range is the set of radial functions on $\Rm^3 \times \Rm$. This transformation satisfies $(\partial_\tau^2 - \Delta_y) \circ \mathbf{T}_2 = \mathbf{T}_2 \circ (\partial_\tau^2 - \Delta_{{\mathbb H}^3} -1)$. A basic calculation shows that if $v_1$ solves \eqref{equation 102}, then $v= \mathbf{T}_2 v_1$ satisfies (CP2). 

\paragraph{Route 2} We have another decomposition $\mathbf{T} = \mathbf{T}_3^{-1} \circ \mathbf{T}_4 \circ \mathbf{T}_3$, where 
\begin{align*}
 &(\mathbf{T}_3 u)(|x|, t) = |x| u(x,t);& &(\mathbf{T}_4 w)(s, \tau) = w(e^\tau \sinh s, t_0 + e^\tau \cosh s).&
\end{align*}
Both $\mathbf{T}_3 u$ and $\mathbf{T}_4 w$ are functions defined on $[0,\infty) \times \Rm$. These two transformations satisfy the commutator identities 
\begin{align*}
 &(\partial_t^2 - \partial_r^2) \circ \mathbf{T}_3 = \mathbf{T}_3 \circ (\partial_t^2 - \Delta_x);& &(\partial_\tau^2 - \partial_s^2) \circ \mathbf{T}_4 = e^{2\tau} \mathbf{T}_4 \circ (\partial_t^2 - \partial_r^2).&
\end{align*}
As a result, if $u$ is a radial solution to (CP1), then $w = \mathbf{T}_3 u$ and $w_1 = \mathbf{T}_4 w$ solve the non-linear wave equations $\partial_t^2 w - \partial_r^2 w = - \frac{1}{r^{p-1}} |w|^{p-1} w$ and $\partial_\tau^2 w_1 - \partial_s^2 w_1 = - e^{-(p-3)\tau} \frac{1}{\sinh^{p-1} s} |w_1|^{p-1} w_1$, respectively. 

\paragraph{The structure of this paper} This paper is organized as follows. In section 2 we collect notations, recall the Strichartz estimates and introduce a local theory for a class of wave equations in the form of $\partial_t^2 u - \Delta u = - \phi(x) e^{-\kappa t} |u|^{p-1} u$ with a function $\phi: \Rm^3 \rightarrow [-1,1]$ and a constant $\kappa \geq 0$. In particular, we combine the energy conservation law with our local theory to conclude that any solution to (CP1) with a finite energy is defined for all time. Next in Section 3 we discuss the global behaviour of solutions to the wave equation above with a suitable coefficient function $\phi(x)$. More precisely, we prove a few global space-time integral estimates if the initial data come with a finite energy, one of which is a Morawetz-type estimate. After all of these preparation work is finished, we prove the main theorem in the last three sections. In Section 4 we start by proving a few preliminary estimates on the solutions $u$ to (CP1). Then we apply the transformation $\mathbf{T}$ and show that $v = \mathbf{T} u$ is indeed a solution to (CP2) in Section 5. In the final section we verify that $v$ has a finite energy, take advantage of the space-time integral estimates we obtained in Section 3, rewrite them in term of $u$ and eventually finish the proof.  

\section{Preliminary Results}

\subsection{Notations}

\paragraph{The $\lesssim$ symbol} We use the notation $A \lesssim B$ if there exists a constant $c$, so that the inequality $A \leq c B$ always holds.  In addition, a subscript of the symbol $\lesssim$ indicates that the constant $c$ is determined by the parameter(s) mentioned in the subscript but nothing else. In particular, $\lesssim_1$ means that the constant $c$ is an absolute constant. 

\paragraph{Radial functions} Let $u(x,t)$ be a spatially radial function. By convention $u(r,t)$ represents the value of $u(x,t)$ when $|x| = r$. 

\paragraph{Linear wave propagation} Given a pair of initial data $(u_0, u_1)$, we define $\mathbf{S}_{L,0}(t) (u_0,u_1)$ to be the solution $u$ of the free linear wave equation $u_{tt} - \Delta u = 0$ with initial data $(u, u_t)|_{t=0} = (u_0, u_1)$. If we are also interested in the velocity $u_t$, we can use the notation 
\begin{align*}
 &\mathbf{S}_L (t) (u_0,u_1) \doteq (u(\cdot, t), u_t(\cdot, t)),& &\mathbf{S}_L (t) \begin{pmatrix} u_0\\ u_1 \end{pmatrix} \doteq \begin{pmatrix} u(\cdot, t)\\ u_t(\cdot, t) \end{pmatrix}.& 
\end{align*}

\subsection{Local theory} 

In this subsection we consider the local theory of the equation 
\begin{equation}
 \left\{\begin{array}{ll} 
 \partial_t^2 v - \Delta v = -\phi(x) e^{-\kappa t} |v|^{p-1} v, & (x,t) \in \Rm^3 \times \Rm;\\
 v(\cdot, t_0) = v_0 \in \dot{H}^1 (\Rm^3); &\\
 v_t (\cdot,t_0) = v_1 \in L^2(\Rm^3). & \end{array} \right. \label{equ3}
\end{equation}
Here $\phi: \Rm^3 \rightarrow [-1,1]$ is a measurable function, $\kappa$ is a nonnegative constant and $p \in [3,5)$. This covers both equations (CP1) and (CP2). 

\begin{definition}
We say that a solution $v$ solves the equation \eqref{equ3} in a time interval $I$ containing $t_0$, if $v$ satisfies 
\begin{itemize} 
 \item $(v(\cdot, t), v_t(\cdot, t)) \in C(I; \dot{H}^1 \times L^2 (\Rm^3))$;
 \item The norm $\|v\|_{L^{2p/(p-3)} L^{2p} (J \times \Rm^3)}$ is finite for any bound closed interval $J \subseteq I$;
 \item The integral equation
\[
 v(\cdot,t) = \mathbf{S}_{L,0}(t-t_0)  (v_0, v_1) + \int_{t_0}^t \frac{\sin ((t-\tau)\sqrt{-\Delta})}{\sqrt{-\Delta}} G(\cdot, \tau, v(\cdot, \tau)) d\tau
\]
holds for all $t \in I$, here $G(x, t, v) =  - \phi(x) e^{-\kappa t} |v|^{p-1}v$.
\end{itemize}
\end{definition}

\paragraph{Strichartz estimates} The basis of our local theory is the following generalized Strichartz estimates. (Please see Proposition 3.1 of \cite{strichartz}, here we use the Sobolev version in $\Rm^3$)

\begin{proposition} \label{strichartz}
Let $2 \leq q_1,q_2 \leq \infty$, $2 \leq r_1, r_2 < \infty$ and $\rho_1, \rho_2, s \in \Rm$ with
\begin{align*}
 &1/{q_i} + 1/{r_i} \leq 1/2, \quad i=1,2;& &&\\
 &1/{q_1} + 3/{r_1} = 3/2 - s' + \rho_1;& &1/{q_2} + 3/{r_2}= 1/2 + s' + \rho_2.&
\end{align*}
Let $v$ be the solution of the following linear wave equation ($t_0 \in I$)
\begin{equation}
\left\{\begin{array}{ll} \partial_t^2 v - \Delta v = F (x,t),  & (x,t)\in \Rm^3 \times I;\\
(v, v_t) |_{t=t_0} = (v_0,v_1) \in \dot{H}^{s'} (\Rm^3) \times \dot{H}^{s'-1}(\Rm^3). \end{array}\right.
\end{equation}
Then there exists a constant independent of $I$ and initial data $(u_0,u_1)$, so that
\begin{align*}
 &\|(v(\cdot, t), v_t(\cdot, t))\|_{C(I;\dot{H}^{s'} \times \dot{H}^{s'-1})} +
 \|D_x^{\rho_1} v\|_{L^{q_1} L^{r_1} (I\times \Rm^3)}\\
 &\qquad  \leq C \left( \|(v_0,v_1)\|_{\dot{H}^{s'} \times \dot{H}^{s'-1}} +
 \|D_x^{-\rho_2} F(x,t)\|_{L^{\bar{q}_2} L^{\bar{r}_2}(I\times \Rm^3)} \right).
\end{align*}
\end{proposition}

\paragraph{A fixed-point argument} We first choose specific coefficients $\rho_1 = \rho_2 = 0$, $s'=1$, $(q_1,r_1) = (2p/(p-3), 2p)$, $(q_2,r_2) = (\infty, 2)$ in the Strichartz estimates
\begin{align*}
 &\left\|(v(\cdot,t), v_t(\cdot,t))\right\|_{C([t_1,t_2]; \dot{H}^1 \times L^2)} + \|v\|_{L^{\frac{2p}{p-3}}L^{2p} ([t_1,t_2]\times \Rm^3)}  \\
 &\qquad\quad \leq C_p \left[\|(v(\cdot, t_1), v_t(\cdot,t_1))\|_{\dot{H}^1 \times L^2} + \|(\partial_t^2 -\Delta) v\|_{L^1 L^2 ([t_1, t_2] \times \Rm^3)}\right],
\end{align*}
and observe the inequalities 
\begin{align*}
 \|G(\cdot, \cdot, v)\|_{L^1 L^2 ([t_1, t_2]\times \Rm^3)} &\leq e^{-\kappa t_1} (t_2 -t_1)^{\frac{5-p}{2}} \|v\|_{L^{\frac{2p}{p-3}}L^{2p} ([t_1,t_2]\times \Rm^3)}^p;\\
 \|G(\cdot, \cdot, v_1) -G(\cdot, \cdot, v_2)\|_{L^1 L^2 ([t_1, t_2]\times \Rm^3)} &\leq  \left[\|v_1\|_{L^{\frac{2p}{p-3}}L^{2p} ([t_1,t_2]\times \Rm^3)}^{p-1} + \|v_2\|_{L^{\frac{2p}{p-3}}L^{2p} ([t_1,t_2]\times \Rm^3)}^{p-1}\right] \\
 & \qquad \times e^{-\kappa t_1} (t_2 -t_1)^{\frac{5-p}{2}} \|v_1 -v_2\|_{L^{\frac{2p}{p-3}}L^{2p} ([t_1,t_2]\times \Rm^3)}.
 \end{align*}
A fixed-point argument then shows (Our argument is similar to a lot of earlier works. See \cite{loc1, ls}, for instance.)
\begin{theorem} [Local solution] \label{local existence 2}
Given a time $t_0$ and a pair $(v_0,v_1) \in \dot{H}^1 \times L^2$, then there is a maximal time interval $(t_0-T_{-}(v_0,v_1,t_0), t_0 + T_{+}(v_0,v_1,t_0))$ in which the equation \eqref{equ3} with the initial condition $(v, v_t)|_{t=t_0} = (v_0,v_1)$ has a unique solution $v(x,t)$. In addition we have
\begin{align*}
 T_+ (v_0,v_1,t_0) > T_1 & \doteq C_1(p) e^{2\kappa t_0/(5-p)} \|(v_0,v_1)\|_{\dot{H}^1 \times L^2 (\Rm^3)}^{-2(p-1)/(5-p)};\\
 \|v(x,t)\|_{L^{2p/(p-3)}L^{2p} ([t_0,t_0 +T_1] \times \Rm^3)}  & \leq C_2(p) \|(v_0,v_1)\|_{\dot{H}^1 \times L^2 (\Rm^3)}.
\end{align*}
\end{theorem}

\begin{remark} \label{d12L4L4}
If $v$ is a solution to \eqref{equ3}, then we have $\|D^{1/2} v\|_{L^4 L^4 ([a,b] \times \Rm^3)} < +\infty$ for any finite bounded interval $[a,b]$ contained in the maximal lifespan of $v$ by the Strichartz estimates. 
\end{remark}
  
\begin{proposition}
Any solution $u$ to (CP1) is global in time, i.e. it has a maximal lifespan $\Rm$. 
\end{proposition}
\begin{proof}
The conservation law of energy guarantees that the norm $\|(u(\cdot,t), u_t(\cdot,t))\|_{\dot{H}^1 \times L^2} \lesssim E^{1/2}$ is uniformly bounded for all time $t$ in the maximal lifespan of $u$. The combination of this fact and Theorem \ref{local existence 2}  implies that $u$ is well-defined for all $t > 0$. Since (CP1) is time-invertible, we are able to conclude that the maximal lifespan of $u$ must be $\Rm$. 
\end{proof}

\paragraph{Perturbation theory} Next let us consider the continuous dependence of the solutions to \eqref{equ3} on the initial data. The special case with $\phi(x) \equiv 1$ and $\kappa =0$ has been proved in Appendix of \cite{shen2}. We can prove the general case in exactly the same way. 

\begin{theorem} \label{perturbation theory 2}
Let $\tilde{v}$ be a solution of equation \eqref{equ3} in a bounded time interval $I$ with initial data $(\tilde{v}_0,\tilde{v}_1)$, so that
\begin{align*}
 &\|(\tilde{v}_0,\tilde{v}_1)\|_{\dot{H}^1 \times L^2} < \infty;& &\|\tilde{v}\|_{L^{2p/(p-3)}L^{2p}(I\times \Rm^3)} < M.&
\end{align*}
There exist two constants $\eps_0 (I,M), C(I,M)>0$, such that if $(v_0,v_1) \in \dot{H}^1 \times L^2$ satisfy 
\[
 \|(v_0 - \tilde{v}_0,v_1- \tilde{v}_1)\|_{\dot{H}^1 \times L^2} < \eps_0 (I,M),
\]
then the corresponding solution $v$ of \eqref{equ3} with initial data $(v_0,v_1)$ is well-defined in $I$ so that
\begin{align*}
 \|v - \tilde{v}\|_{L^{2p/(p-3)}L^{2p} (I \times \Rm^3)} & \leq C(I,M) \|(v_0 - \tilde{v}_0,v_1- \tilde{v}_1)\|_{\dot{H}^1 \times L^2};\\
 \left\|\begin{pmatrix} v(\cdot, t) \\ v_t(\cdot,t)\end{pmatrix} - \begin{pmatrix} \tilde{v}(\cdot, t)\\ \tilde{v}_t (\cdot, t)\end{pmatrix} \right\|_{C(I; \dot{H}^1 \times L^2)} & \leq
 C(I,M) \|(v_0 - \tilde{v}_0,v_1- \tilde{v}_1)\|_{\dot{H}^1 \times L^2}.
\end{align*}
\end{theorem}
  
\section{A Wave Equation with a Time Dependent Nonlinearity}

In this section we discuss the global behaviour of the solutions to the equation 
\begin{equation}
 \left\{ \begin{array}{ll} v_{tt} - \Delta u = - \phi(x) e^{-\kappa t} |v|^{p-1}v, & (x, t) \in \Rm^3 \times \Rm;\\
 v(\cdot, t_0) = v_0 \in \dot{H}^1(\Rm^3)\cap L^{p+1} (\Rm^3; \phi(x) dx); &\\
 v_t(\cdot, t_0) = v_1 \in L^2 (\Rm^3). &
 \end{array} \right. \label{equ4}
\end{equation}
Here we assume that $p \in [3,5)$, $\kappa \geq 0$ are constants and $\phi: \Rm^3 \rightarrow [0,1]$ is a measurable function. The equation (CP2) corresponds to the case with $\kappa = p-3$ and $\phi(x) = \left(\frac{|x|}{\sinh |x|}\right)^{p-1}$. In this case the parameter $\kappa>0$ whenever $p>3$. 

\subsection{Monotonicity of the Energy}
 
Now let us consider the ``energy'' defined by 
\[
 E(t) = \int_{\Rm^3} \left[\frac{1}{2}|\nabla_x v(x, t)|^2 + \frac{1}{2}|v_t(x,t)|^2 + e^{-\kappa t} \phi(x) \frac{|v(x,t)|^{p+1}}{p+1}\right] dx.
\]
If $u$ is sufficiently smooth and decays sufficiently fast near infinity, we can differentiate and obtain  
\begin{align*}
 E'(t) = & \int_{\Rm^3} \left[\nabla v \nabla v_t + v_t v_{tt} + e^{-\kappa t} \phi(x) |v|^{p-1} v v_t -\kappa e^{-\kappa t} \phi \frac{|v|^{p+1}}{p+1}\right] dx \\
=  & \int_{\Rm^3} v_t \left(-\Delta v + v_{tt} + e^{-\kappa t} \phi |v|^{p-1} v\right)dx - \frac{\kappa}{p+1} \int_{\Rm^3} e^{-\kappa t} \phi |v|^{p+1} dx\\
 = & - \frac{\kappa}{p+1} \int_{\Rm^3} e^{-\kappa t} \phi(x) |v(x,t)|^{p+1} dx \leq 0. 
\end{align*}
One can verify that this formula of $E'(t)$ works for general solutions $v$ of the Cauchy problem \eqref{equ4} as well by standard smooth approximation and cut-off techniques. Therefore we have 

\begin{proposition} \label{damping} 
Let $v$ be a solution to the Cauchy problem \eqref{equ4} in a time interval $[t_0, t_0 +T_+)$ with $E(t_0) < \infty$. 
\begin{itemize}
 \item If $\kappa >0$, then $E(t)$ is a non-increasing function of $t \in [t_0,t_0 +T_+)$. In addition, we have the integral estimate 
\[
 \int_{t_0}^{t_0 + T_+}\!\! \int_{\Rm^3} e^{-\kappa t} \phi(x) |v(x,t)|^{p+1} dx\, dt \leq \frac{p+1}{\kappa} E(t_0). 
\] 
 \item If $\kappa=0$, then $E(t)$ is a constant independent of $t$. 
\end{itemize}
\end{proposition}

\subsection{Global behaviour in the positive time direction}

Assume that $v$ is a solution to the Cauchy problem \eqref{equ4} with a maximal lifespan $(t_0-T_-, t_0+T_+)$. Given any $t \in I_+ \doteq [t_0, t_0 +T_+)$, Proposition \ref{damping} implies 
\[
 \|(v(\cdot, t), v_t(\cdot, t))\|_{\dot{H}^1 \times L^2} \leq \left[2 E(t)\right]^{1/2} \leq \left[2E(t_0)\right]^{1/2}. 
\] 
According to Theorem \ref{local existence 2}, this means that there are two positive constants $T_1$ and $N_1$, such that if $t \in I_+$, then we have $[t, t + T_1] \subseteq I_+$ and $\|v\|_{L^{2p/(p-3)} L^{2p} ([t, t+T_1])} \leq N_1$.  It immediately follows that $T_+ = +\infty$. Namely the solution $u$ is defined for all time $t > t_0$. Furthermore, if $\kappa>0$ we have 
\begin{align*}
 \left\|G(x,t,v)\right\|_{L_t^1 L_x^2 ([t_0, \infty) \times \Rm^3)} = & \sum_{j=0}^\infty \left\|e^{-\kappa t}\phi(x)|v|^{p-1} v\right\|_{L^1 L^2 ([t_0+jT_1, t_0 + (j+1)T_1] \times \Rm^3)}\\
 = & \sum_{j=0}^\infty e^{-\kappa t_0 - j\kappa T_1} T_1^{(5-p)/2} \|v\|_{L^{2p/(p-3)} L^{2p} ([t_0+jT_1, t_0 + (j+1)T_1] \times \Rm^3)}^p\\
 = & \sum_{j=0}^\infty e^{-\kappa t_0 - j\kappa T_1} T_1^{(5-p)/2}  N_1^p < \infty.
\end{align*}
Recalling the Strichartz estimates and the fact that the linear wave propagation preserves the $\dot{H}^1 \times L^2$ norm, we obtain 
\begin{align*}
 & \lim_{t_1, t_2 \rightarrow +\infty} \left\|\mathbf{S}_L(-t_1) \begin{pmatrix} v(\cdot, t_1)\\ v_t(\cdot,t_1)\end{pmatrix} - \mathbf{S}_L(-t_2) \begin{pmatrix} v(\cdot, t_2)\\ v_t(\cdot,t_2)) \end{pmatrix} \right\|_{\dot{H}^1 \times L^2}\\
 = & \lim_{t_1, t_2 \rightarrow +\infty} \left\|\mathbf{S}_L(t_2 -t_1) \begin{pmatrix} v(\cdot, t_1)\\ v_t(\cdot,t_1)\end{pmatrix} - \begin{pmatrix} v(\cdot, t_2)\\ v_t(\cdot,t_2)) \end{pmatrix} \right\|_{\dot{H}^1 \times L^2}\\
 \leq & \lim_{t_1, t_2 \rightarrow +\infty} \|G(x,t,v)\|_{L_t^1 L_x^2 ([t_1, t_2]\times \Rm^3)} = 0.
\end{align*}
As a result, the pair $\mathbf{S}_L(-t) (v(\cdot,t), v_t(\cdot,t))$ converges in the space $\dot{H}^1 \times L^2$ as $t \rightarrow \infty$. Let us assume $\mathbf{S}_L(-t) (v(\cdot,t), v_t(\cdot,t)) \rightarrow (v_0^+, v_1^+)$. This is equivalent to saying
\[
 \lim_{t \rightarrow +\infty} \left\| (v(\cdot,t), v_t(\cdot,t)) - \mathbf{S}_L(t) (v_0^+, v_1^+) \right\|_{\dot{H}^1 \times L^2} = 0. 
\]
We summarize our results below 

\begin{theorem} [Global behaviour] \label{global behaviour}
 Let $v$ be a solution to the Cauchy problem \eqref{equ4} with a finite energy $E(t_0) < \infty$. Then $v$ is well-defined for all $t \geq t_0$. If we also have $\kappa>0$, then there exists a pair $(v_0^+, v_1^+) \in \dot{H}^1 \times L^2$ so that 
\[
 \lim_{t \rightarrow \infty} \left\|(v(\cdot,t), v_t(\cdot,t)) - \mathbf{S}_L (t) (v_0^+, v_1^+)\right\|_{\dot{H}^1 \times L^2} = 0. 
\]
\end{theorem}

\noindent A combination of Theorem \ref{global behaviour} and Proposition \ref{damping} immediately gives

\begin{corollary} \label{global corollary} Let $v$ be a solution to the Cauchy problem \eqref{equ4} with $\kappa>0$ and a finite energy $E(t_0) < \infty$. Then we have 
\[
 \int_{t_0}^{\infty} \int_{\Rm^3} e^{-\kappa t} \phi(x) |v(x,t)|^{p+1} dx\, dt \leq \frac{p+1}{\kappa} E(t_0). 
\] 
\end{corollary}

\subsection{A Morawetz-type Inequality}

\begin{proposition} \label{Morawetz}
Let $v$ be a solution to the Cauchy problem \eqref{equ4} in a time interval $[t_0, t_0 +T_+)$ so that 
\begin{itemize}
 \item [(I)] $E(t_0) < \infty$;
 \item [(II)] The inequalities $0 \leq \phi(x) \leq 1$ and $(p-1) \phi - x \cdot \nabla \phi \geq 0$ hold for all $x \in \Rm^3$.
\end{itemize}  
Then we have the following Morawetz-type inequality 
\[
 \int_{t_0}^{t_0 + T_+} \!\!\int_{\Rm^3}  e^{-\kappa t} \cdot  \frac{(p-1) \phi - x\cdot \nabla \phi}{|x|} \cdot |v|^{p+1}  dx\, dt \lesssim_1 E(t_0). 
\]
\end{proposition}

\paragraph{Outline of the proof} Let us consider a function $a(x)=|x|$ and define 
\[
 M(t) = \int_{\Rm^3} v_t(x,t) \left(\nabla v(x,t) \cdot \nabla a(x) + \frac{1}{2} \Delta a(x) v(x,t)\right)dx. 
\]
A basic calculation shows 
\begin{align*}
 &\nabla a = \frac{x}{|x|},& &\Delta a = \frac{2}{|x|},& &\mathbf{D}^2 a \geq 0,& &\Delta \Delta a \leq 0.&
\end{align*}
As a result, we obtain an upper bound on $|M(t)|$ by Hardy's inequality $\|v/|x|\|_{L^2} \lesssim \|\nabla v\|_{L^2}$:  
\begin{equation} \label{upper bound of M}
 \left|M(t)\right| \leq  \|v_t(\cdot, t)\|_{L^2} \left(\|\nabla v(\cdot,t)\|_{L^2} + \|v(x,t)/|x|\|_{L_x^2(\Rm^3)}\right) \lesssim_1 E(t).
\end{equation}
Next we calculate the derivative $M'(t)$ informally 
\begin{align*}
 M'(t)  = & \int_{\Rm^3} v_{tt} \left(\nabla v \cdot \nabla a + \frac{1}{2} v \Delta a\right)dx + \int_{\Rm^3} v_t \left(\nabla v_t \cdot \nabla a + \frac{1}{2} v_t \Delta a\right)dx\\
 = & \int_{\Rm^3} \Delta v \left(\nabla v \cdot \nabla a + \frac{1}{2} v \Delta a\right)dx - \int_{\Rm^3} \phi(x) e^{-\kappa t} |v|^{p-1} v \left(\nabla v \cdot \nabla a + \frac{1}{2} v \Delta a\right) dx\\
 &\qquad  +  \int_{\Rm^3} v_t \left(\nabla v_t \cdot \nabla a + \frac{1}{2} v_t \Delta a\right)dx\\
 = & I_1 + I_2 + I_3. 
\end{align*}
Let us start with $I_1$. For simplicity we use lower indices to represent partial derivatives. 
\begin{align*}
 I_1 = & \int_{\Rm^3} \left(\sum_{i,j=1}^3 v_{ii} v_j a_j\right) dx - \frac{1}{2} \int_{\Rm^3} |\nabla v|^2 \Delta a \,dx  - \frac{1}{2} \int_{\Rm^3} v \nabla v\cdot  \nabla \Delta a\, dx\\
 = & - \int_{\Rm^3} \left(\sum_{i,j=1}^3 a_{ij} v_i v_j\right) dx - \int_{\Rm^3} \left(\sum_{i,j=1}^3 a_{j} v_i v_{ij}\right) dx - \frac{1}{2} \int_{\Rm^3} |\nabla v|^2 \Delta a\, dx + \frac{1}{4} \int_{\Rm^3} |v|^2 \Delta \Delta a\, dx\\
 \leq & - \frac{1}{2} \int_{\Rm^3} \nabla a \cdot \nabla (|\nabla v|^2)dx - \frac{1}{2} \int_{\Rm^3} |\nabla v|^2 \Delta a\, dx\\
 = & 0. 
\end{align*}
Here we use the facts $\mathbf{D}^2 a \geq 0$ and $\Delta \Delta a \leq 0$. In addition we have 
\begin{align*}
 I_2 = & - \frac{1}{p+1} \int_{\Rm^3} \phi(x) e^{-\kappa t} \nabla (|v|^{p+1}) \cdot \nabla a\, dx 
 -\frac{1}{2} \int_{\Rm^3} \phi(x) e^{-\kappa t} |v|^{p+1} \Delta a\, dx\\
 = & \frac{1}{p+1} \int_{\Rm^3}  e^{-\kappa t} |v|^{p+1} \nabla \phi \cdot \nabla a dx + \left( \frac{1}{p+1} -\frac{1}{2}\right) \int_{\Rm^3}  e^{-\kappa t} |v|^{p+1} \phi \Delta a\, dx \\
 = & \frac{1}{p+1} \int_{\Rm^3}  e^{-\kappa t} |v|^{p+1} \left(\nabla \phi \cdot \nabla a - \frac{p-1}{2} \phi \Delta a \right)dx\\
 = & \frac{-1}{p+1} \int_{t_1}^{t_2} \!\!\int_{\Rm^3}  e^{-\kappa t} \cdot  \frac{(p-1) \phi - x\cdot \nabla \phi}{|x|} \cdot |v|^{p+1}  dx\, dt.
\end{align*}
Finally 
\begin{equation*}
 I_3 =  \frac{1}{2} \int_{\Rm^3} \nabla (|\partial_t v|^2) \cdot \nabla a\, dx + \frac{1}{2} \int_{\Rm^3} |\partial_t v|^2 \Delta a\, dx = 0. 
 \end{equation*}
Now we collect all the terms above and then integrate from $t=t_1$ to $t=t_2$: 
\begin{align*}
 M(t_2) - M(t_1)  \leq & \frac{-1}{p+1} \int_{t_1}^{t_2} \!\!\int_{\Rm^3}  e^{-\kappa t} \cdot  \frac{(p-1) \phi - x\cdot \nabla \phi}{|x|} \cdot |v|^{p+1}  dx\, dt . 
\end{align*}
We plug the upper bound on $|M(t)|$ as given in \eqref{upper bound of M} into the left hand side above, recall the monotonicity of $E(t)$ and finally complete our proof. 
\begin{remark}
 The argument above works only for solutions $v$ that satisfies certain regularity conditions. However, Proposition \ref{Morawetz} still holds for all solutions $v$ with a finite energy $E(t_0) < \infty$. This can be proved via standard smooth approximation and cut-off techniques. Please refer to Section 4 of \cite{subhyper} for more details about this type of argument. 
\end{remark}

\subsection{An Equivalent Condition of Scattering} \label{sec: equivalent condition of scattering}

Let us start by a technical result. 

\begin{proposition} \label{low index regularity}
Let $v$ be a solution to the Cauchy problem \eqref{equ4} in a bounded closed time interval $I = [a,b]$ with initial data $(v_0,v_1) \in (\dot{H}^1 \cap \dot{H}^{s_p}) \times (L^2 \cap \dot{H}^{s_p-1})$. Then we have $(v(\cdot,t), v_t(\cdot,t)) \in C(I; \dot{H}^{s_p} \times \dot{H}^{s_p-1})$ and 
\[ 
   \|D^{s_p-1/2} v\|_{L^4 L^4 ([a,b] \times \Rm^3)} < +\infty.
\]
\end{proposition}
\begin{proof}
Let us recall the Strichartz estimate
\begin{align*}
 \|(v(\cdot,t), v_t(\cdot,t))\|_{C(I; \dot{H}^{s_p} \times \dot{H}^{s_p-1})} & + \|D^{s_p-1/2} v\|_{L^4 L^4 ([a,b] \times \Rm^3)} \\
 & \lesssim \|(v_0,v_1)\|_{\dot{H}^{s_p} \times \dot{H}^{s_p-1}} + \|(\partial_t^2-\Delta) v\|_{L^{\frac{2}{1+s_p}} L^{\frac{2}{2-s_p}} (I \times \Rm^3)}.
\end{align*}
As a result, it suffices to show 
\begin{equation} \label{to prove 1}
  \left\|-e^{-\kappa t} \phi(x) |v|^{p-1} v\right\|_{L^{\frac{2}{1+s_p}} L^{\frac{2}{2-s_p}} (I \times \Rm^3)} < \infty \Leftrightarrow 
  \left\|\phi^{1/p} v\right\|_{L^{\frac{4(p-1)p}{5p-9}} L^{\frac{4(p-1)p}{p+3}} (I \times \Rm^3)} < \infty. 
\end{equation}
On one hand, the monotonicity of $E(t)$ implies 
\[
 \sup_{t \in I} \int_{\Rm^3} e^{-\kappa t} \phi(x) |v(x,t)|^{p+1} dx\, dt < \infty \Rightarrow \left\|\phi^{1/p} v\right\|_{L^\infty L^{p+1} (I \times \Rm^3)} < \infty. 
\]
On the other hand, the Strichartz estimates give
\[
 \|v\|_{L^5 L^{10} (I \times \Rm^3)} < \infty \Longrightarrow \left\|\phi^{1/p} v\right\|_{L^5 L^{10} (I \times \Rm^3)} < \infty. 
\]
We combine these two inequalities via an interpolation (with ratio $(5-p)(2p+3)(p+1):5(p-3)(3p+1)$) to obtain 
\[
 \left\|\phi^{1/p} v\right\|_{L^\frac{2p(p-1)(9-p)}{(p-3)(3p+1)} L^\frac{4(p-1)p}{p+3} (I \times \Rm^3)} < +\infty.
\]
This is a sufficient condition of \eqref{to prove 1} because $I$ is a finite interval and $\frac{2p(p-1)(9-p)}{(p-3)(3p+1)} \geq \frac{4(p-1)p}{5p-9}$.
\end{proof}

\begin{proposition}[Scattering with a finite $L^{2(p-1)} L^{2(p-1)}$ norm] \label{L2p2}
 Let $u$ be a solution to (CP1) with initial data $(u_0,u_1) \in (\dot{H}^1 \cap \dot{H}^{s_p}) \times (L^2 \cap \dot{H}^{s_p-1})$. If $\|u\|_{L^{2(p-1)} L^{2(p-1)}(\Rm \times \Rm^3)} < \infty$, then $u$ scatters in both two time directions. More precisely, there exist two pairs $(u_0^\pm, u_1^\pm) \in (\dot{H}^1 \cap \dot{H}^{s_p}) \times (L^2 \cap \dot{H}^{s_p-1})$, so that the following limit holds for each $s' \in [s_p,1]$
\[
 \lim_{t \rightarrow \pm \infty} \left\|(u(\cdot,t), u_t(\cdot, t)) - \mathbf{S}_L(t) (u_0^\pm, u_1^\pm) \right\|_{\dot{H}^{s'} \times \dot{H}^{s'-1} (\Rm^3)} = 0. 
\]
\end{proposition}
\begin{proof}
 Since the equation is time-invertible, it suffices to consider the case $t \rightarrow +\infty$. In the argument below, we temporarily assume that $s'$ is either $1$ or $s_p$.  We start by picking up an arbitrary finite time interval $[a,b]$ and applying the Strichartz estimates 
 \begin{align*}
  &\|D_x^{s'-1/2} u\|_{L^{4} L^{4}([a,b]\times \Rm^3)}\\
   &\quad \leq C \|(u(\cdot,a), u_t(\cdot,a))\|_{\dot{H}^{s'} \times \dot{H}^{s'-1}} + C \|D_x^{s'-1/2} (-|u|^{p-1}u)\|_{L^{4/3} L^{4/3} ([a,b]\times \Rm^3)}\\
  &\quad \leq C \|(u(\cdot,a), u_t(\cdot,a))\|_{\dot{H}^{s'} \times \dot{H}^{s'-1}} + C_{s',p} \|u\|_{L^{2(p-1)} L^{2(p-1)}([a,b] \times \Rm^3)}^{p-1} \|D_x^{s'-1/2} u\|_{L^{4} L^{4}([a,b]\times \Rm^3)}.
 \end{align*}
 In the last step above, we apply the chain rule with fractional derivatives. Please see Lemma 2.5 of \cite{kenig1} and the citation therein for more details. By the assumption $\|u\|_{L^{2(p-1)} L^{2(p-1)}(\Rm\times \Rm^3)} < \infty$, we can fix a large number $a$, so that $C_{s',p} \|u\|_{L^{2(p-1)} L^{2(p-1)}([a,\infty)\times \Rm^3)}^{p-1} < 1/2$. We plug this upper bound into the inequality above, recall the fact $\|D_x^{s'-1/2} u\|_{L^{4} L^{4}([a,b]\times \Rm^3)} < \infty$
 that comes from either Remark \ref{d12L4L4}, if $s'=1$, or Proposition \ref{low index regularity}, if $s'=s_p$, and obtain 
 \[
  \|D_x^{s'-1/2} u\|_{L^{4} L^{4}([a,b]\times \Rm^3)} < 2 C \|(u(\cdot,a), u_t(\cdot,a))\|_{\dot{H}^{s'} \times \dot{H}^{s'-1}} < \infty.
 \]
Here the finiteness of $\dot{H}^{s'}\times \dot{H}^{s'-1}$norm comes from either the definition of a solution, if $s'=1$, or Proposition \ref{low index regularity}, if $s'=s_p$. Please note that the upper bound here does not depend on the right endpoint $b$. A combination of this uniform upper bound with the fact that $\mathbf{S}_L(t)$ preserves the $\dot{H}^{s'} \times \dot{H}^{s'-1}$ norm implies
 \begin{align*}
 &  \limsup_{t_1, t_2 \rightarrow +\infty} \left\|\mathbf{S}_{L}(-t_2) \begin{pmatrix} u(\cdot, t_2)\\ u_t(\cdot,t_2)\end{pmatrix} - \mathbf{S}_L (-t_1) \begin{pmatrix} u(\cdot,t_1)\\ u_t(\cdot, t_1)\end{pmatrix}\right\|_{\dot{H}^{s'} \times \dot{H}^{s'-1}} \\
  = & \limsup_{t_1, t_2 \rightarrow +\infty} \left\|\begin{pmatrix} u(\cdot, t_2)\\ u_t(\cdot,t_2)\end{pmatrix} - \mathbf{S}_L (t_2-t_1) \begin{pmatrix} u(\cdot,t_1)\\ u_t(\cdot, t_1)\end{pmatrix}\right\|_{\dot{H}^{s'} \times \dot{H}^{s'-1}} \\
  \leq & C  \limsup_{t_1, t_2 \rightarrow +\infty}  \|D_x^{s'-1/2} (-|u|^{p-1}u)\|_{L^{4/3} L^{4/3} ([t_1,t_2]\times \Rm^3)}\\
  \leq & C_{s',p}   \limsup_{t_1, t_2 \rightarrow +\infty}  \left(\|u\|_{L^{2(p-1)} L^{2(p-1)}([t_1,t_2] \times \Rm^3)}^{p-1} \|D_x^{s'-1/2} u\|_{L^{4} L^{4}([t_1,t_2]\times \Rm^3)}\right) = 0.
 \end{align*}
As a result,  the pair $\mathbf{S}_L (-t) (u(\cdot, t), u_t(\cdot,t))$ converges in the space $\dot{H}^{s'} \times \dot{H}^{s'-1} (\Rm^3)$ as $t \rightarrow +\infty$. Since the argument above works for both $s'=1$ and $s'=s_p$, we know that there exists a pair $(u_0^+, u_1^+) \in (\dot{H}^1 \cap \dot{H}^{s_p}) \times (L^2 \cap \dot{H}^{s_p-1})$ so that the limit 
 \[
  \lim_{t \rightarrow +\infty} \left\|\mathbf{S}_L (-t) (u(\cdot, t), u_t(\cdot,t)) - (u_0^+, u_1^+)\right\|_{\dot{H}^{s'} \times \dot{H}^{s'-1} (\Rm^3)} = 0
 \]
holds for $s' \in \{1, s_p\}$. By a basic interpolation the limit above holds for all $s' \in [s_p,1]$. This is equivalent to our conclusion
 \[
  \lim_{t \rightarrow +\infty} \left\|(u(\cdot, t), u_t(\cdot,t)) - \mathbf{S}_L (t)  (u_0^+, u_1^+)\right\|_{\dot{H}^{s'} \times \dot{H}^{s'-1} (\Rm^3)} = 0. 
 \]
 \end{proof}

\section{Preliminary Estimates on Solutions}

\begin{lemma} \label{lm1} (See also Lemma 6.12 of \cite{subhyper} for the 2D version) Let $u$ be a solution to the linear wave equation
\[
 \left\{\begin{array}{l} \partial_t^2 u - \Delta u = F(x,t), \,\,\,\, (x,t)\in \Rm^3 \times [0,T];\\
u |_{t=0} = u_0; \\
\partial_t u |_{t=0} = u_1;\end{array}\right.
\]
with radial data $u_0, u_1$ and $F$. These data satisfy the inequalities
\begin{align*}
 &|u_0(x)| \leq A_1 |x|^{-1-\alpha}, \quad |F(x,t)| \leq B_1 |x|^{-3} (|x|-t)^{-\beta},& & \hbox{if}\; |x| > R;&\\
 &\int_{|x|>R} |x|^{1+2\alpha} |u_1(x)|^2 dx \leq A_1^2;& &&
\end{align*}
with constants $R, A_1, B_1> 0$ and  $0 < \alpha, \beta < 1/2$. Then there exists a constant $C = C (\alpha, \beta)\geq 1$ such that the solution $u$ satisfies
\[
 |u(x,t)| \leq C |x|^{-1}\left[A_1 (|x|-t)^{-\alpha} +  B_1 (|x|-t)^{-\beta} \right], \; \hbox{if}\; t\in [0,T]\;\hbox{and}\; |x| > R + t.
\]
\end{lemma}

\begin{remark}
 In the proof of Lemma \ref{lm1} (as well as Corollary \ref{cor1} below) we always assume that $u$ is sufficiently smooth. Otherwise we can apply standard smooth approximation techniques. 
\end{remark}

\begin{proof} 
Let us consider the function $w: \Rm^+ \times [0,T] \rightarrow \Rm$ defined by the formula $w(r,t) = r u(r,t)$. One can check that the function $w$ satisfies the following wave equation defined on $R \times [0,T]$
\[
 \partial_t^2 w - \partial_r^2 w = r F(r,t).
\]
An explicit formula for the solution to a one-dimensional wave equation shows that
\begin{align}
 w(r_0,t_0) = & \frac{1}{2} \left[w(r_0-t_0, 0) + w(r_0+t_0, 0)\right] + \frac{1}{2} \int_{r_0-t_0}^{r_0+t_0} \partial_t w(r,0) dr \nonumber \\
  & \qquad + \frac{1}{2} \int_{0}^{t_0} \int_{r_0-t_0+t}^{r_0+t_0-t} r F(r,t)\, dr dt, \label{1D formula}
\end{align}
whenever $r_0 > t_0+ R$ and $t_0 \in [0,T]$. Our assumptions on $F$ and the initial data $u_0$, $u_1$ give the upper bounds 
\begin{align*}
 &|w(r_0-t_0,0)| \leq A_1 (r_0-t_0)^{-\alpha};&  &|w(r_0+t_0,0)| \leq A_1 (r_0+t_0)^{-\alpha};& &r F(r,t) \leq B_1 r^{-2} (r-t)^{-\beta};&
\end{align*} 
and
\begin{align*}
\left| \int_{r_0-t_0}^{r_0+t_0} \partial_t w(r,0) dr\right| & = \left| \int_{r_0-t_0}^{r_0+t_0} r u_1(r) dr \right|\\
 & \leq  \left(\int_{r_0-t_0}^{r_0+t_0} r^{-1-2\alpha} dr\right)^{1/2} \left(\int_{r_0-t_0}^{r_0+t_0} r^{3+2\alpha} |u_1(r)|^2 dr\right)^{1/2}\\
 & \lesssim_\alpha (r_0-t_0)^{-\alpha} \left(\int_{|x| > r_0-t_0} |x|^{1+2\alpha} |u_1(x)|^2 dx\right)^{1/2}\\
 & \leq A_1 (r_0-t_0)^{-\alpha}.
\end{align*}
We then plug the upper bounds above into the identity \eqref{1D formula} and obtain 
\begin{align*}
 |w(r_0,t_0)| & \leq \frac{A_1}{2} \left[(r_0 - t_0)^{-\alpha} + (r_0+t_0)^{-\alpha}\right] + \frac{1}{2} \left| \int_{r_0-t_0}^{r_0+t_0} \partial_t w(r,0) dr\right| \\
 & \qquad \qquad + \frac{B_1}{2} \int_{0}^{t_0} \int_{r_0-t_0+t}^{r_0+t_0-t} r^{-2} (r-t)^{-\beta}\,  dr dt\\
 & \leq C_{\alpha} A_1 (r_0 - t_0)^{-\alpha} + \frac{B_1}{2} \int_{r_0-t_0}^{r_0+t_0} \int_s^{(r_0 + t_0 +s)/2} r^{-2} s^{-\beta}\, dr ds\\
 & \leq C_{\alpha} A_1 (r_0 - t_0)^{-\alpha} + \frac{B_1}{2} \int_{r_0-t_0}^{r_0+t_0}  s^{-1-\beta} ds\\
 & \leq C_{\alpha} A_1 (r_0 - t_0)^{-\alpha} + C_{\beta} B_1 (r_0 - t_0)^{-\beta}.
\end{align*}
Here we deal with the double integral by the change of variables $(r,s) = (r, r-t)$. Finally we recall $w = ru$, divide both sides of the inequality above by $r_0$ and finish the proof. 
\end{proof}

\begin{proposition} \label{pointwise estimate}
Assume $3\leq p <5$. Let $(u_0,u_1)$ and $A, \eps$ be initial data and positive constants as in Theorem \ref{main1}. Fix any constant $\delta < \min\{\eps, 1/{10}\}$. Then there exist constants $B_1= B_1 (\delta) >0$ and $R= R (\delta,\eps, A)> 1$, such that the solution $u$ to (CP1) with initial data $(u_0, u_1)$ satisfies 
\begin{equation} \label{es1} 
 |u(x,t)| \leq B_1 |x|^{-1} (|x|-t)^{-\delta},\;\;\; \hbox{if}\; t\geq 0\; \hbox{and}\; |x|> t + R.
\end{equation}
\end{proposition}
\begin{proof}
Let $C = C(\delta, 3\delta)$ be the constant as in the conclusion of Lemma \ref{lm1}. We can always find two small positive constants $A_1 = A_1 (\delta)$ and $B_1 = B_1 (\delta) < 1$, such that
\[
 B_1 > C (A_1 + B_1^3).
\]
By Remark \ref{integral in r}, Remark \ref{point-wise es and energy} and the assumption $\delta < \eps$, we can always find a large constant $R = R(A, \eps, \delta) > 1$, such that if $|x| > R$, then
\begin{align*}
 &|u_0(x)| < A_1 |x|^{-1-\delta};& &\int_{|x|>R} |x|^{1+2\delta} |u_1(x)|^2 dx < A_1^2.&
\end{align*}
We claim that these constants $B_1$ and $R$ work. In fact, If $t_1$ is sufficiently small, then the restriction of solution $u$ to the time interval $[0,t_1]$ can be obtained by a fixed-point argument according to our local theory. More precisely, if we set $\tilde{u}_{0} \equiv 0$ and define
\[
 \tilde{u}_{n+1} (\cdot, t) = \mathbf{S}_{L,0} (t) (u_0, u_1) + \int_{0}^t \frac{\sin ((t-\tau)\sqrt{-\Delta})}{\sqrt{-\Delta}} F(\tilde{u}_{n} (\cdot, \tau)) d\tau,
\]
where $F(u) = - |u|^{p-1}u$, then we have 
\[
 \lim_{n \rightarrow \infty} \|\tilde{u}_n - u\|_{L^{\frac{2p}{p-3}} L^{2p} ([0,t_1] \times \Rm^3)} = 0. 
\]
An induction argument immediately follows:
\begin{itemize}
 \item [(I)] The function $\tilde{u}_0$ satisfies the inequality \eqref{es1} if $t \in [0,t_1]$; 
 \item [(II)] If $\tilde{u}_n$ satisfies \eqref{es1} for $t \in [0,t_1]$, then we have 
 \[
  |F(\tilde{u}_n(x,t))| = \left|B_1 |x|^{-1} (|x|-t)^{-\delta}\right|^p \leq B_1^3 |x|^{-3} (|x|-t)^{-3\delta}, \quad \hbox{if}\; |x|>t+R \; \hbox{and}\; 0 \leq t \leq t_1.     
 \]
Thus we can apply Proposition \ref{lm1} and obtain 
 \begin{align*}
   \left|\tilde{u}_{n+1} (x,t)\right| \leq &  C(\delta, 3 \delta) |x|^{-1} \left[ A_1 (|x|-t)^{-\delta} + B_1^3 (|x|-t)^{-3\delta} \right] \\
   \leq  & C(A_1+B_1^3) |x|^{-1} (|x|-t)^{-\delta} \\
   \leq & B_1 |x|^{-1} (|x|-t)^{-\delta},
 \end{align*}
 whenever $t \in [0,t_1]$ and $|x| > t+ R$. 
\end{itemize}
In summary, $\tilde{u}_n$ satisfies \eqref{es1} for all $n \geq 0$ and $t \in [0,t_1]$. Passing to the limit, we conclude that $u$ satisfies \eqref{es1} for $t \in [0,t_1]$. In order to generalize this to all time $t \in [0,T]$ we only need to iterate our argument above. More details about this ``double induction'' argument can be found in Proposition 6.16 of the author's joint work \cite{subhyper} with G. Staffilani.  
\end{proof}

\begin{corollary} \label{cor1}
 Let $(u_0,u_1)$ be initial data as in Theorem \ref{main1} and $A$, $\eps$, $\delta$, $B_1$, $R$ be constants associated to it as above. Then there exist a function $f: [R,\infty) \rightarrow \Rm$ with 
 \begin{align*}
    \int_R^\infty s^{1+\delta} |f(s)|^2 \, ds \lesssim_{A, \eps, \delta} 1 
 \end{align*}
 so that for all $t\geq 0$ and $r> t + R$ the function $w(r,t) = r u(r,t)$ satisfies 
 \begin{align}
 &|w_t (r,t)+ w_r(r,t)| \leq f (r+t);& &|w_t(r,t) - w_r(r,t)| \leq f (r-t).& \label{ineq1} 
\end{align}
\end{corollary}

\begin{proof}
 For simplicity we define $z_1(r,t) = w_t(r,t) + w_r(r,t)$ and $z_2(r,t) = w_t(r,t) - w_r(r,t)$. Since $z_1, z_2$ satisfy the identities 
 \begin{align*}
  \frac{\partial}{\partial s} [z_1 (r +t -s, s)] & = (r +t -s) F(r +t -s, s);\\
  \frac{\partial}{\partial s} [z_2 (r -t +s, s)] & = (r -t +s) F(r -t +s, s);
 \end{align*}
 where the function $F$ is defined as $F(r,t) = -|u(r,t)|^{p-1} u(r,t)$, we can integrate from $s = 0$ to $s= t$ by the fundamental theorem of calculus 
 \begin{align*}
 & z_1(r, t) = z_1 (r +t, 0) + \int_{0}^{t} (r +t -s) F(r +t-s, s) ds;\\
 & z_2 (r, t) = z_2 (r -t, 0) + \int_{0}^{t} (r -t +s) F(r -t+s, s) ds.
 \end{align*}
 Next we rewrite $z_1(r+t,0), z_2(r-t,0)$ in term of $u_0, u_1$ by their definition and obtain 
 \begin{align*}
 & z_1(r,t) =  (r+t) \left[u_1(r+t) + \partial _r u_0 (r+t)\right] + u_0(r+t) + \int_{0}^{t} (r +t -s) F(r +t-s, s)\, ds; \\
 & z_2(r,t)  = (r-t) \left[u_1(r-t) - \partial_r u_0 (r-t)\right]  - u_0(r-t) +  \int_{0}^{t} (r -t +s) F(r -t+s, s)\, ds.
 \end{align*}
We claim that we can choose $f (s) = s |u_1(s)| + s |\partial_r u_0 (s)| + C s^{-1-\delta}$ for a suitable constant $C = C(A, \eps, \delta)$. It follows Remark \ref{integral in r}, the point-wise estimate $u(r,t) \lesssim r^{-1}(r-t)^{-\delta}$ and a couple of estimates on the integrals in the expression of $z_1, z_2$ . For the first integral we have
 \begin{align*}
  \left|\int_{0}^{t} (r +t -s) F(r +t-s, s)\, ds \right| \lesssim & \int_0^t (r+t-s)\left[(r+t-s)^{-1} (r+t-2s)^{-\delta}\right]^3 ds\\
  \lesssim & (r+t)^{-2} \int_0^t (r+t-2s)^{-\delta}\, ds \\
  \lesssim & (r+t)^{-1-\delta}. 
 \end{align*}
 The second integral can be dealt with in a similar way
  \begin{align*}
  \left|\int_{0}^{t} (r -t+s) F(r -t+s, s)\, ds \right| \lesssim & \int_0^t (r -t+s)\left[(r-t+s)^{-1} (r-t)^{-\delta}\right]^3 ds\\
  \lesssim & (r-t)^{-\delta} \int_0^t (r-t+s)^{-2}\, ds \\
  \lesssim & (r-t)^{-1-\delta}. 
 \end{align*}

 \end{proof}  

\section{A transformation} \label{sec: transformation}

Let $u(x,t)$ be a global and radial solution to (CP1). We consider the function $v = \mathbf{T} u$ defined by 
\[
 v(y, \tau) = \frac{\sinh |y|}{|y|} e^\tau u\left( e^\tau \frac{\sinh |y|}{|y|}\cdot y, t_0 + e^\tau \cosh |y|\right), \quad (y, \tau) \in \Rm^3 \times \Rm.
\]
Here $t_0$ is a negative number to be determined later. This transformation can be rewritten in the form of $(\mathbf{T} u)(y,\tau) =  \frac{\sinh |y|}{|y|} e^\tau u (\mathbf{\tilde{T}}(y, \tau))$, where the geometric transformation $\mathbf{\tilde{T}}: \Rm^3\times \Rm \rightarrow \{(x,t) \in \Rm^3 \times \Rm: t-t_0 > |x|\}$ is defined by 
\[
 \mathbf{\tilde{T}}(y, \tau) = \left( e^\tau \frac{\sinh |y|}{|y|}\cdot y, t_0 + e^\tau \cosh |y|\right).
\]
In particular, $\mathbf{\tilde{T}}$ maps the hyperplane $\tau = \tau_0$ in the $y$-$\tau$ space-time to the upper sheet of the hyperboloid $(t-t_0)^2 - |x|^2 = e^{2\tau_0}$ in the $x$-$t$ space-time. 

\paragraph{Radial expression} The function $v$ is still a radial function and can be given in term of polar coordinates $(s, \Theta, \tau) \in [0,\infty) \times {\mathbb S}^2 \times \Rm$ by
\[
 v(s, \Theta, \tau) = \frac{\sinh s}{s} e^\tau u(e^\tau \sinh s \cdot \Theta, t_0+ e^\tau \cosh s).
\]
For simplicity we can omit $\Theta$ and write 
\[
   v(s,\tau) = \frac{\sinh s}{s} e^\tau u(e^\tau \sinh s, t_0 + e^\tau \cosh s).
\] 
\paragraph{Differentiation} Let us recall that the function $w(r,t) = ru(r,t)$  satisfies the equation $w_{tt} - w_{rr} = - r|u|^{p-1} u$, we can rewrite the function $sv(s, \tau)$ in the form of 
\[
 s v(s, \tau) = w(e^\tau \sinh s, t_0 + e^\tau \cosh s). 
\] 
A simple calculation shows 
\begin{align}
 &(sv)_\tau =  (e^\tau \sinh s) w_r + (e^\tau \cosh s) w_t;&  &(sv)_s = (e^\tau \cosh s) w_r + (e^\tau \sinh s) w_t. &  \label{derivative of sv}
\end{align}
The values of $w_r$ and $w_t$ here are taken at the point $(e^\tau \sinh s, t_0 + e^\tau \cosh s)$. Next we can differentiate again and obtain \footnote{Here we temporarily assume that the functions involved are sufficiently smooth. Otherwise we can apply the standard smoothing approximation techniques}. 
\begin{align*}
 (sv)_{\tau \tau} = & (e^\tau \sinh s) w_r + (e^\tau \sinh s)^2 w_{rr} + (e^\tau \sinh s)(e^\tau \cosh s) w_{rt}\\
  & + (e^\tau \cosh s) w_t + (e^\tau \cosh s)(e^\tau \sinh s) w_{tr} + (e^\tau \cosh s)^2 w_{tt};\\
  (sv)_{ss} = & (e^\tau \sinh s) w_r + (e^\tau \cosh s)^2 w_{rr} + (e^\tau \cosh s)(e^\tau \sinh s)w_{rt}\\
  & + (e^\tau \cosh s) w_t + (e^\tau \sinh s)(e^\tau \cosh s) w_{tr} + (e^\tau \sinh s)^2 w_{tt}.
\end{align*}
Therefore we have (let us recall $r = e^\tau \sinh s$)
\begin{align*}
 v_{\tau \tau} - v_{ss} - \frac{2}{s} v_s = & \frac{1}{s} \left[(sv)_{\tau \tau} - (sv)_{ss}\right] = \frac{e^{2\tau}}{s} \left[w_{tt} - w_{rr}\right] 
 =  -\frac{e^{2\tau}}{s} r |u|^{p-1} u \\
 = & -\left(\frac{s}{\sinh s}\right)^{p-1} e^{-(p-3)\tau} \left|\frac{\sinh s}{s} e^\tau u\right|^{p-1} \frac{\sinh s}{s} e^\tau u\\
 = & -\left(\frac{s}{\sinh s}\right)^{p-1} e^{-(p-3)\tau} |v|^{p-1} v. 
\end{align*}
In other words, $v(y,\tau)$ satisfies the non-linear wave equation 
\[
 v_{\tau \tau} - \Delta_y v = - \left(\frac{|y|}{\sinh |y|}\right)^{p-1} e^{-(p-3)\tau} |v|^{p-1}v, \qquad (\tau, y) \in \Rm \times \Rm^3. \quad (CP3)
\]
Finally a basic calculation gives the following change of variables formula for integrals of radial functions
\begin{equation}
 dx\, dt = 4\pi r^2 dr\, dt = 4\pi e^{4\tau} \sinh^2 s \,ds \, d\tau = e^{4\tau} \left(\frac{\sinh |y|}{|y|}\right)^2 dy\, d\tau. \label{change of variables}
\end{equation}

\section{Proof of the Main Theorem}
 
Let us consider a solution $u$ to (CP1) as given in Theorem \ref{main1} with the constants $A, \eps$. We first fix a number $\delta = \min\{\eps/2, 1/10\}$ and let $B_1$, $R$ be the constants as given in  Proposition \ref{pointwise estimate}. Please note that all these constants $\delta$, $B_1$ and $R$ are determined solely by $A$ and $\eps$. 
Next we fix a negative time $t_0 = - \sqrt{R^2 +1} -1$ and perform the transformation $v = \mathbf{T} u$ as described in the previous section. We claim 

\begin{lemma} \label{finite energy}
 There exists a time $\tau \in [-1,0]$, so that the energy 
\begin{align*}
 E(\tau) & = \int_{\Rm^3} \left[\frac{1}{2}|\nabla_y v(y, \tau)|^2 + \frac{1}{2}|v_\tau(x,\tau)|^2 + e^{-(p-3) \tau} \left(\frac{|y|}{\sinh |y|}\right)^{p-1} \frac{|v(y,\tau)|^{p+1}}{p+1}\right] dy \\
 &< C(A, \eps). 
\end{align*}
Here $C(A, \eps)$ is a finite constant determined solely by the constants $A$ and $\eps$. 
\end{lemma}
\begin{remark}
 This actually means that $E(0) < C(A, \eps, p) < \infty$. 
\end{remark}

\subsection{Proof of Lemma \ref{finite energy}} 

First of all, we observe that 
\begin{align*}
 \int_{\Rm^3} e^{-(p-3) \tau} \left(\frac{|y|}{\sinh |y|}\right)^{p-1} \frac{|v(y,\tau)|^{p+1}}{p+1} dy & \lesssim_1  \left\|\left(\frac{|y|}{\sinh |y|}\right)^{p-1}\right\|_{L^{6/(5-p)}(\Rm^3)}\|v(\cdot, \tau)\|_{L^6(\Rm^3)}^{p+1}\\
 & \lesssim_1 \|v(\cdot, \tau)\|_{\dot{H}^1 (\Rm^3)}^{p+1}
\end{align*}
Therefore it suffices to show that 
\[
 E_0 (\tau)  = \int_{\Rm^3} \left[\frac{1}{2}|\nabla_y v(y, \tau)|^2 + \frac{1}{2}|v_\tau(x,\tau)|^2 \right] dy < C' (A, \eps, p).
\]
Next we use the fact that $v$ is radial and rewrite $E_0(\tau)$ in term of polar coordinates 
\begin{align*}
 E_0 (\tau) & = \int_0^\infty 2\pi \left[|v_s (s, \tau)|^2 + |v_\tau (s, \tau)|^2\right] s^2 ds.
\end{align*}
We split the integral into two parts: the integral over $[0, s_0(\tau)]$ and the integral over $[s_0(\tau), + \infty)$. 
\[
  E_0 (\tau) = \int_0^{s_0(\tau)} + \int_{s_0(\tau)}^\infty \doteq  E_0^{(1)} (\tau) + E_0^{(2)} (\tau).
\]
The radius $s_0(\tau) \doteq \cosh^{-1} (-t_0 e^{-\tau}) > \cosh^{-1} \sqrt{2}$ corresponds to the value of time $t = t_0 + e^\tau \cosh s_0 = 0$. 

\paragraph{Large radius part} In this case we have $t = t_0 + e^\tau \cosh s \geq 0$ and 
\[
 r - t  = e^\tau \sinh s - (t_0 + e^\tau \cosh s) = -t_0 - e^\tau e^{-s} \geq - t_0 - e^\tau e^{-s_0} = \sqrt{t_0^2 -e^{2\tau}} > R. 
\]
Therefore we have 
\begin{itemize}
 \item[(i)] $t_0 + e^\tau e^s = r + t \simeq r = e^\tau \sinh s \simeq e^\tau e^s$; 
 \item[(ii)] we can apply the inequalities regrading $u$, $w_r$, $w_t$ we obtained in Proposition \ref{pointwise estimate} and Corollary \ref{cor1} to obtain 
\begin{align}
 &|w_t + w_r | \leq f(t_0 + e^\tau e^s);&  &|w_t - w_r | \leq f(-t_0 - e^{\tau -s});& \label{ineq pm} \\
 &|u| \lesssim_{A, \eps} (e^\tau \sinh s)^{-1}\; \Longrightarrow \; |v| \lesssim_{A,\eps} s^{-1}.& && \label{ineq uv}
\end{align}
 All the values of $u$, $w_r$ and $w_t$ are taken at the point $(r,t) = (e^\tau \sinh s, t_0 + e^\tau \cosh s)$. 
\end{itemize}
We combine the identities \eqref{derivative of sv} with the inequalities \eqref{ineq pm} and obtain 
\begin{align*}
 2|(sv)_\tau| &= 2\left|(e^\tau \sinh s) w_r + (e^\tau \cosh s) w_t\right| = e^\tau \left|e^s (w_t + w_r) + e^{-s} (w_t - w_r)\right|\\
 & \leq e^{\tau +s} f(t_0 + e^\tau e^s) + e^{\tau -s} f(-t_0 - e^{\tau -s});\\
 2\left|(sv)_s\right| & = 2\left|(e^\tau \cosh s) w_r + (e^\tau \sinh s) w_t\right| = e^\tau \left|e^s (w_t + w_r) - e^{-s} (w_t - w_r)\right|\\
 & \leq e^{\tau +s} f(t_0 + e^\tau e^s) + e^{\tau -s} f(-t_0 - e^{\tau -s}).
\end{align*}
A basic calculation shows 
\begin{align*}
 E_0^{(1)} (\tau) & = 2 \pi \int_{s_0(\tau)}^\infty \left[|v_s (s, \tau)|^2 + |v_\tau (s, \tau)|^2\right] s^2 ds \\
 & \leq 2\pi \int_{s_0(\tau)}^\infty \left[\left|\frac{\partial (sv)}{\partial s}(s, \tau) - v(s,\tau)\right|^2 + \left|\frac{\partial (sv)}{\partial \tau} (s, \tau)\right|^2\right] ds\\
 & \leq 4\pi  \int_{s_0(\tau)}^\infty \left[\left|(sv)_s\right|^2 + \left|(sv)_\tau\right|^2 + v^2 \right] ds.
 \end{align*}
By the upper bounds on $|(sv)_\tau|$, $|(sv)_s|$, $|v|$ given above we finally obtain a universal upper bound on $E_0^{(1)}(\tau)$:
\begin{align*}
 E_0^{(1)} (\tau) & \lesssim_{A,\eps} \int_{s_0(\tau)}^\infty e^{2\tau +2s} \left|f(t_0 + e^\tau e^s)\right|^2 ds  + \int_{s_0(\tau)}^\infty e^{\tau -s} \left|f(-t_0 - e^{\tau -s})\right|^2 ds 
 + \int_{s_0(\tau)}^\infty s^{-2} ds\\
 & \lesssim \int_R^\infty \bar{r} \left|f(\bar{r})\right|^2 d\bar{r} + \int_R^{-t_0} \left|f(\tilde{r})\right|^2 d\tilde{r} + 1\\
 & \lesssim_{A,\eps} 1. 
\end{align*}
Here we need to apply the change of variables $\bar{r} = t_0 + e^\tau e^s > R$, $\tilde{r} = -t_0 - e^{\tau -s} >R$ and use the estimate (i).
In the final step we use the assumption on the function $f$ in Corollary \ref{cor1}
\[
 \int_R^\infty \bar{r}^{1+\delta} |f(\bar{r})|^2 d\bar{r} \lesssim_{A, \eps} 1. 
\] 

\begin{figure}[h]
 \centering
 \includegraphics[scale=0.85]{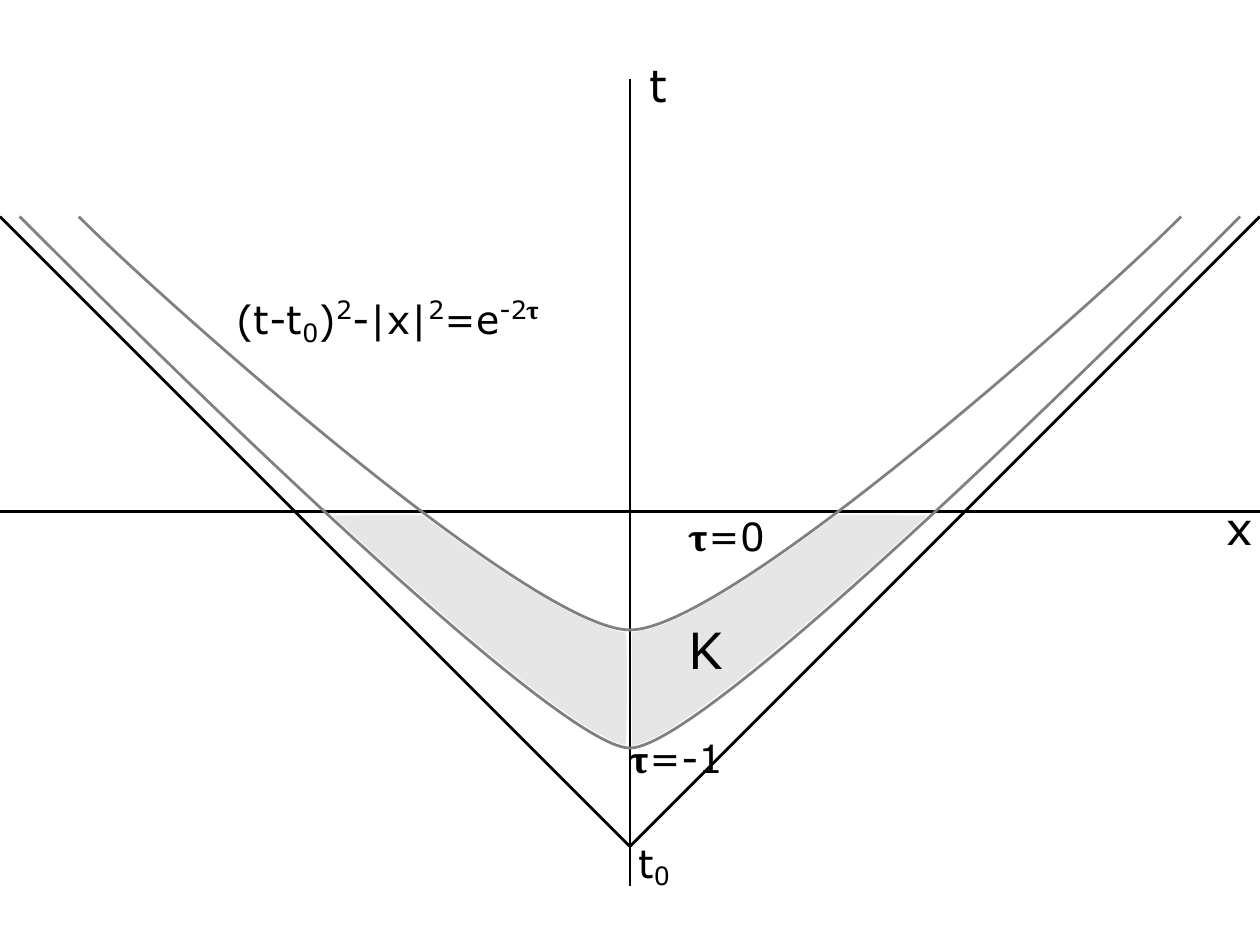}
 \caption{Illustration of region $K$} \label{f1}
\end{figure}

\paragraph{Small radius part}  Now we need to consider the upper bound of $\displaystyle \inf_{\tau \in [-1,0]} E_0^{(2)} (\tau)$, which can be dominated by an integral  
\begin{align*}
 \inf_{\tau \in [-1,0]} E_0^{(2)} (\tau) \leq & 2 \pi \int_{-1}^{0} \int_0^{s_0(\tau)} \left[|v_s (s, \tau)|^2 + |v_\tau (s, \tau)|^2\right] s^2 ds\, d\tau \\
 = & \frac{1}{2} \int_{-1}^0  \int_0^{s_0(\tau)} e^{-4\tau} \left(\frac{s}{\sinh s}\right)^2 \left[|v_s (s, \tau)|^2 + |v_\tau (s, \tau)|^2\right] 4\pi e^{4\tau} \sinh^2 s \,ds\,d\tau.
\end{align*}
Let us recall our definition of $v$ and differentiate: 
\begin{align*}
 v_{\tau} & = \frac{\sinh s}{s} e^\tau u + \frac{\sinh s \cosh s}{s} e^{2\tau} u_t + \frac{\sinh^2 s}{s} e^{2\tau} u_r; \\
 v_s & = \frac{s \cosh s -\sinh s}{s^2} e^\tau u + \frac{\sinh^2 s}{s} e^{2\tau} u_t + \frac{\sinh s \cosh s}{s} e^{2 \tau} u_r. 
\end{align*}
As a result we have 
\begin{align*}
 \left|\frac{s}{\sinh s} v_\tau\right| & \leq \left\{|u| + (t -t_0) |u_t| + r |u_r|\right\}_{(r,t) = (e^\tau \sinh s, t_0 + e^\tau \cosh s)};\\
 \left|\frac{s}{\sinh s} v_s \right| & \leq \left\{|u| + r |u_t| + (t -t_0) |u_r|\right\}_{(r,t) = (e^\tau \sinh s, t_0 + e^\tau \cosh s)}.
\end{align*}
Using these upper bounds and the change of variables formula \eqref{change of variables}, we obtain 
\begin{align*}
 \inf_{\tau \in [-1,0]} E_0^{(2)} (\tau) & \lesssim \iint_K (1+t_0^2) \left(|u_t|^2 + |\nabla u|^2 + |u|^2\right) dx\, dt\\
& \lesssim (1+t_0^2) \int_{-t_0}^0 \int_{B(0, |t_0|)} \left(|u_t|^2 + |\nabla u|^2 + |u|^{p+1} + 1\right) dx\, dt \\
& \lesssim (1+t_0^2) t_0^4 + (1+t_0^2)t_0 \tilde{E} \lesssim_{A, \eps} 1. 
\end{align*}
Here the region $K = \{(x,t): e^{-2} \leq (t-t_0)^2 - |x|^2 \leq 1, t_0 < t \leq 0\} \subseteq B(0, |t_0|) \times [-t_0, 0]$, as illustrated in figure \ref{f1}. The letter $\tilde{E}$ represents the energy of solution $u$, whose upper bound has been given in Remark \ref{point-wise es and energy}. Combining the small radius part with the large radius part, we have 
\[
 \inf_{t \in [-1,0]} E_0(\tau) \leq \sup_{t \in [-1,0]} E_0^{(1)}(\tau) + \inf_{t \in [-1,0]} E_0^{(2)}(\tau) \lesssim_{A, \eps} 1
\]
thus finish the proof of Lemma \ref{finite energy}.

\subsection{A global integral estimate} 

Now $v$ is a radial solution to (CP2) with a finite energy $E(0) \lesssim_{A, \eps} 1$.  We claim 
\begin{equation} \label{comb_ineq}
 I' \doteq \int_{0}^{\infty} \!\!\int_{\Rm^3}  e^{-(p-3) \tau} \left(\frac{|y|}{\sinh |y|}\right)^{p-1} |v(y,\tau)|^{2(p-1)}  dy\, d\tau \lesssim_{A, \eps, p} 1. 
\end{equation}  
\begin{proof} 
First of all, Proposition \ref{Morawetz} gives a Morawetz-type estimate  
\[
 \int_{0}^{\infty} \!\!\int_{\Rm^3}  e^{-(p-3) \tau} \frac{|y|^{p-1} \cosh |y|}{\sinh^p |y|} |v(y,\tau)|^{p+1}  dy\, d\tau \lesssim_1 E(0) \lesssim_{A, \eps} 1. 
\]
Since $v(\cdot, \tau)$ is a radial $\dot{H}^1(\Rm^3)$ function, we also have 
\[
  \left| v(y, \tau)\right| \lesssim \frac{\|v(\cdot, \tau)\|_{\dot{H}^1(\Rm^3)}}{|y|^{1/2}} \lesssim  \frac{(E(\tau))^{1/2}}{|y|^{1/2}} \lesssim_{A, \eps} \left(\frac{\cosh |y|}{\sinh |y|}\right)^{1/2}.
\]
A combination of these two inequalities gives 
\begin{equation} \label{Morawetz inequality}
 \int_{0}^{\infty} \!\!\int_{\Rm^3}  e^{-(p-3) \tau} \left(\frac{|y|}{\sinh |y|}\right)^{p-1} |v(y,\tau)|^{p+3}  dy\, d\tau \lesssim_{A, \eps} 1.
\end{equation}
If $p=3$, this is exactly the same inequality as \eqref{comb_ineq}. On the other hand, if  $p \in (3,5)$, then we are able to apply Proposition \ref{global corollary} and obtain another integral estimate
\begin{equation} \label{energy decay inequality}
 \int_{0}^{\infty} \int_{\Rm^3} e^{-(p-3) \tau} \left(\frac{|y|}{\sinh |y|}\right)^{p-1} |v(y,\tau)|^{p+1} dy\, d\tau  \leq \frac{p+1}{p-3} E(0) \lesssim_{A, \eps, p} 1. 
\end{equation}
Finally we can apply an interpolation between the inequalities \eqref{Morawetz inequality} and \eqref{energy decay inequality} to conclude the proof, because our assumption $p\in (3,5)$ implies that $p+1 < 2(p-1) < p+3$. 
\end{proof}

\subsection{Completion of the proof for the main theorem}

We have already known that the solution is well-defined for all time $t \in \Rm$. According to Proposition \ref{L2p2}, it suffices to show
\[
 I \doteq \int_0^\infty \int_{\Rm^3} |u(x,t)|^{2(p-1)} dx\, dt \lesssim_{A, \eps, p} 1.
\]
We first break the integral into two parts 
\begin{align*}
 I & = \int_0^\infty \int_{|x|>t+R} |u(x,t)|^{2(p-1)} dx\, dt + \int_{0}^\infty \int_{|x|<t+R} |u(x,t)|^{2(p-1)} dx\, dt\\
  & \leq \int_0^\infty \int_{|x|>t+R} |u(x,t)|^{2(p-1)} dx\, dt + \iint_{\Omega} |u(x,t)|^{2(p-1)} dx\, dt \doteq I_1 + I_2.
\end{align*}
Here the region $\Omega = \{(x,t): |x|^2 < (t-t_0)^2 -1, t>t_0\}$ satisfies (Please see figure \ref{f2})
\begin{itemize}
 \item $\Omega$ contains the region $\{(x,t): |x| < t + R, t \geq 0\}$;
 \item $\Omega$ corresponds to the positive-time part of the $y$-$\tau$ space-time. In other words we have $\Omega = \mathbf{\tilde{T}}(\{(y,\tau): \tau>0\})$. 
\end{itemize} 
\begin{figure}
 \centering 
 \includegraphics[scale=0.85]{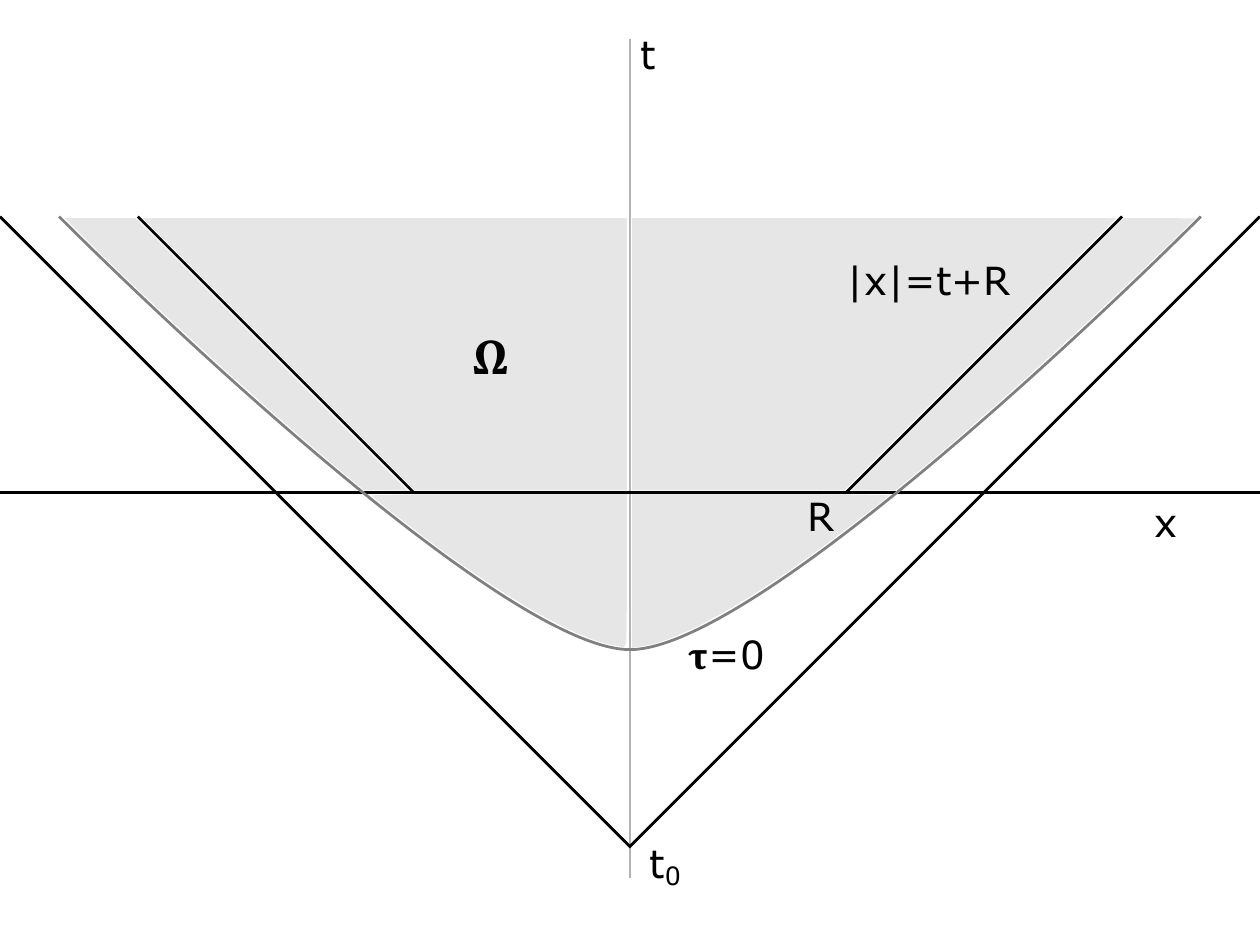}
 \caption{Illustration of the region $\Omega$} \label{f2}
\end{figure}
It is clear that $I_1 \lesssim_{A, \eps, p} 1$ since the inequality $u(r,t) \lesssim_{A, \eps} r^{-1} (r-t)^{-\delta}$ (when $r > t+R$ and $t\geq 0$) implies that
\[
 I_1 \lesssim_{A, \eps} \int_{0}^{\infty}\!\! \int_{t+R}^\infty [r^{-1} (r-t)^{-\delta}]^{2(p-1)} r^2\, dr dt = \int_R^{\infty} \int_s^\infty r^{-2(p-2)} s^{-2(p-1)\delta} \, dr ds \lesssim_{A,\eps} 1.
\]
In order to deal with $I_2$ we apply the change of variables formula \eqref{change of variables}. 
\begin{align*}
 I_2 = & \int_{0}^\infty \int_{\Rm^3} \left(e^{-\tau} \frac{|y|}{\sinh |y|}\right)^{2(p-1)} \left|\frac{\sinh |y|}{|y|} e^\tau u(\mathbf{\tilde{T}} (y,s)) \right|^{2(p-1)} 
 \cdot e^{4\tau} \left(\frac{\sinh |y|}{|y|}\right)^2 dy\, d\tau\\
 = & \int_{0}^\infty \int_{\Rm^3} e^{-2(p-3)\tau} \left(\frac{|y|}{\sinh |y|}\right)^{2p -4} |v(y,\tau)|^{2(p-1)} dy\, d\tau. 
\end{align*}
The last expression of $I_2$ is different from the left hand of \eqref{comb_ineq} (i.e. the integral $I'$) only in the first two exponents. A simple comparison shows that $I_2 \leq I' \lesssim_{A, \eps, p} 1$. This finishes the proof of our main theorem.

\end{document}